\documentclass[options]{amsart}

\usepackage[utf8x]{inputenc}
\usepackage{latexsym,ifthen,amssymb}
\usepackage{latexsym,ifthen,amssymb}
\usepackage[toc,page,title,titletoc,header]{appendix}
\usepackage{multirow}
 \usepackage{longtable}
 \usepackage{rotating}
 \usepackage{graphicx, amsmath, amsthm, amssymb,float,setspace,color, multirow}
\usepackage{enumitem}
\usepackage{graphicx}
\usepackage{bm}
\usepackage{subfigure}
\usepackage{float}
\usepackage{rotating}
\usepackage{epstopdf}
\usepackage{latexsym,ifthen,amssymb}
\usepackage{algorithm}
\usepackage{algorithmicx}
\usepackage{algpseudocode}
\usepackage[toc,page,title,titletoc,header]{appendix}
\usepackage{multirow}
 \usepackage{longtable}
 \usepackage{rotating}
 \usepackage{graphicx, amsmath, amsthm, amssymb,float,setspace,color, multirow}
\usepackage{tikz}
\usepackage{hyperref}
\usepackage{enumitem}
\usepackage[margin=1cm]{caption}

\hypersetup{
	pdfborder={0 0 0}
}
\usetikzlibrary{decorations.markings}

\usepackage{url}
\usepackage[margin=1in]{geometry}
\usepackage{xcolor}	
\usepackage{soul}

\newcommand{\imp}{\rightarrow}

\newcommand{\Nb}{\mathbb{N}}

\newcommand{\Psf}{\mathsf{P}}

\newcommand{\Ical}{\mathcal{I}}
\newcommand{\Jcal}{\mathcal{J}}

\newcommand{\Pcal}{\mathcal{P}}

\newcommand{\uh}{{\upharpoonright}}
\newcommand{\uhr}{{\upharpoonright}}

\renewcommand{\setminus}{\smallsetminus}

\newcommand{\dbf}{\mathbf{d}}

\newcommand{\s}[1]{\ensuremath{\sf{#1}}}

\newcommand{\ovw}[2]{\s{OVW}(#1,#2)}
\newcommand{\vw}[2]{\s{VW}(#1,#2)}

\DeclareMathOperator{\fut}{\s{FUT}}
\DeclareMathOperator{\wfut}{\s{wFUT}}
\DeclareMathOperator{\rca}{\s{RCA}_0}
\DeclareMathOperator{\piooca}{\Pi^1_1\s{CA}}
\DeclareMathOperator{\aca}{\s{ACA}}
\DeclareMathOperator{\wkl}{\s{WKL}}

\DeclareMathOperator{\drt}{\s{DRT}}

\DeclareMathOperator{\coh}{\s{COH}}

\def\mcal{\mathcal}
\def\t{\tilde}

\newtheorem{theorem}{Theorem}[section]
\newtheorem{lemma}[theorem]{Lemma}

\newtheorem{claim}[theorem]{Claim}
\newtheorem{corollary}[theorem]{Corollary}

\theoremstyle{definition}
\newtheorem{definition}[theorem]{Definition}
\newtheorem{statement}[theorem]{Statement}

\theoremstyle{remark}

\newtheorem {question}[theorem]{Question}
\numberwithin{equation}{section}

\newtheoremstyle{noparens}%
  {}{}%
{}{}%
{\bfseries}{.}%
{ }%
{\thmname{#1}\thmnumber{ #2}\thmnote{ #3}}
\theoremstyle{noparens}
\newtheorem*{question*}{Question}
\newtheorem*{theorem*}{Theorem}

\title{A computable analysis of variable words theorems}

\author{Lu Liu}
\address{Department of Mathematics\\
Central South University\\
ChangSha 410083\\
People’s Republic of China}
\email{g.jiayi.liu@gmail.com}

\author{Benoit Monin}
\address{LACL, Département d'Informatique \\
Faculté des Sciences et Technologie \\
61 avenue du Général de Gaulle \\
94010 Créteil Cedex}
\email{benoit.monin@computability.fr}

\author{Ludovic Patey}
\address{Institut Camille Jordan\\
Université Claude Bernard Lyon 1\\
43 boulevard du 11 novembre 1918\\
F-69622 Villeurbanne Cedex}
\email{ludovic.patey@computability.fr}

\thanks{}

\def\t{\tilde}

\def\mcal{\mathcal}

\begin{document}

\begin{abstract}
The Carlson-Simpson lemma is a combinatorial statement occurring in the proof of the Dual Ramsey theorem. Formulated in terms of variable words, it informally asserts that given any finite coloring of the strings, there is an infinite sequence with infinitely many variables such that for every valuation, some specific set of initial segments is homogeneous. Friedman, Simpson, and Montalban asked about its reverse mathematical strength. We study the computability-theoretic properties and the reverse mathematics of this statement, and relate it to the finite union theorem. In particular, we prove the Ordered Variable word for binary strings in $\aca_0$. 
\end{abstract}

\maketitle

\section{Introduction}

Let $(\Nb)^k$ and $(\Nb)^\infty$ denote the set of partitions of $\Nb$ into exactly $k$ and infinitely many non-empty pieces, respectively. For $X \in (\Nb)^\infty$, $(X)^k$ is the set of all $Y \in (\Nb)^k$ which are coarser than $X$.

\begin{statement}[Dual Ramsey theorem]
	$\drt^k$ is the statement ``If $(\Nb)^k$ is colored with finitely many Borel colors, then there is some $X \in (\Nb)^\infty$ such that $(X)^k$ is monochromatic''.
\end{statement}

The Dual Ramsey theorem was proven by Carlson and Simpson~\cite{Carlson1984dual},
and studied from a reverse mathematical viewpoint by Slaman~\cite{Slaman1997note}, Miller and Solomon~\cite{Miller2004Effectiveness} and Dzhafarov et al.~\cite{Dzhafarov2017Effectiveness}.
In this paper, we shall focus on a combinatorial lemma used by Carlson and Simpson to prove the Dual Ramsey theorem. This lemma can be formulated in terms of \emph{variable words}.

\begin{definition}[Variable word]
An \emph{infinite variable word} on a finite alphabet $A$
is an $\omega$-sequence $W$ of elements of $A \cup \{x_i : i \in \Nb\}$
in which all variables occur at least once, and finitely often. Moreover, the first occurrence of $x_i$ comes before the first occurrence of $x_{i+1}$.
A \emph{finite variable word} is an initial segment of an infinite variable word. A finite or infinite variable word is \emph{ordered} if moreover all occurences of $x_i$ come before any occurrence of $x_{i+1}$.
Given $\bar a = a_0a_1 \dots a_{k-1} \in A^{<\omega}$, we let $W(\bar a)$
denote the finite $A$-string obtained by replacing $x_i$ with $a_i$ in $W$ and then truncating the result
just before the first occcurence of $x_k$.
\end{definition}

\begin{statement}[Variable word theorem]
$\vw{n}{r}$ is the statement ``If $A^{<\omega}$ is colored with $r$ colors for some alphabet $A$
of cardinality $n$, there exists
an infinite variable word $W$ such that $\{W(\bar a) : \bar a \in A^{<\omega}\}$ is monochromatic.
$\ovw{n}{r}$ is the same statement as $\vw{n}{r}$ but for ordered variable words.
\end{statement}

In this paper, we study the computability-theoretic properties of the variable word theorems using the framework of reverse mathematics.\footnote{The authors thank Damir Dzhafarov, Stephen Flood, Reed Solomon and Linda Brown Westrick for bringing the attention of the authors to the Carlson-Simpson lemma, and for  numerous discussions. The authors are also thankful to Denis Hirschfeldt and Barbara Csima for showing them how to use Lovasz Local Lemma to prove lower bounds to combinatorial theorems.}

\subsection{Reverse mathematics}
Reverse mathematics is a vast foundational program aiming to determine the optimal axioms to prove ordinary theorems. It uses the framework of second-order arithmetic, with a base theory $\rca$ consisting of the axioms of Robinson arithmetic, the $\Sigma^0_1$ induction scheme and the $\Delta^0_1$ comprehension scheme. The system $\rca$ arguably captures \emph{computable mathematics}. Starting from a proof-theoretic perspective, modern reverse mathematics tends to be seen as a framework to analyse the computability-theoretic features of theorems. Among the distinguished statements, let us mention weak K\"onig's lemma ($\wkl$), asserting that every infinite binary tree has an infinite path, the arithmetic comprehension axiom ($\aca$), and the $\Pi^1_1$ comprehension axiom ($\piooca$), consisting of the comprehension scheme restricted to arithmetic and $\Pi^1_1$ formulas, respectively. See Simpson~\cite{Simpson2009Subsystems} for reference book on classical reverse mathematics.

The statements studied within this framework are mainly of the form $(\forall X)[\Phi(X) \rightarrow (\exists Y)\Psi(X, Y)]$, where $\Phi$ and $\Psi$ are arithmetic formulas with set parameters, and can be considered as \emph{problems}. Given a statement $\Psf$ of this form,
	a set $X$ such that $\Phi(X)$ holds is an \emph{instance} of $\Psf$, and a set $Y$ such that $\Psi(X, Y)$ holds is a \emph{solution} to the $\Psf$-instance $X$. In this paper, we shall consider exclusively statements of this kind.

Friedman and Simpson~\cite{Friedman2000Issues}, and later Montalban~\cite{Montalban2011Open}, asked about the reverse mathematical strength of the ordered variable word. The statement $\ovw{k}{\ell}$ is known to be provable in $\rca + \piooca$. Our main result is a direct combinatorial proof of $\ovw{2}{\ell}$ in $\rca + \aca$.

\begin{theorem}\label{thm:aca-ovw2}
	For every $\ell \geq 2$, $\rca + \aca \vdash \ovw{2}{\ell}$.
\end{theorem}

On the lower bound hand, Miller and Solomon~\cite{Miller2004Effectiveness} constructed a computable instance $c$ of $\ovw{2}{2}$ with no $\Delta^0_2$ solution, and deduced that $\rca + \wkl$ does not prove $\vw{2}{2}$. Indeed, seeing the instance $c$ of $\ovw{2}{2}$ as an instance of $\vw{2}{2}$, and noticing that the jump of a solution to $\vw{2}{2}$ gives a solution to $\ovw{2}{2}$, one can deduce that $c$ has no low $\vw{2}{2}$-solution. In this paper, we improve their lower bound by constructing a computable instance of $\ovw{2}{2}$ whose solutions are of DNC degree relative to $\emptyset'$.

\subsection{Organization of the paper}
In section~\ref{sect:ht-ovw}, we shall give a simple proof of the ordered variable word for binary strings ($\ovw{2}{\ell}$) using the finite union theorem. Then, in section~\ref{sect:ovw-aca}, we provide a direct combinatorial proof of the same statement over $\rca + \aca$.
Finally, in section~\ref{sect:ovw-lower-bounds}, we give a new lower bound on the strength of $\ovw{2}{\ell}$ using a computable version of Lovasz Local Lemma.

\subsection{Notation}

Given two sets $A$ and $B$, we write $A < B$ for the formula $(\forall x \in A)(\forall y \in B) x < y$.
Given a set $A$, we write $A^{<\omega}$ for the set of finite $A$-valued strings. In particular, $2^{<\omega}$ is the set of binary strings. We denote by $\Pcal_{fin}(\Nb)$ the collection of finite \emph{non-empty} subsets of $\Nb$. Given two strings $\sigma, \tau \in A^{<\omega}$, $\sigma*\tau$ denotes their concatenation. We may also write $\sigma\tau$ when there is no ambiguity. Given a string or a sequence $X$ and some $n \in \omega$, we write $X \uhr n$ for the initial segment of $X$ of length $n$. In particular, $X \uhr 0$ is the empty string, written $\varepsilon$.

\section{A simple proof of the Ordered Variable Word theorem from the Finite Union Theorem}\label{sect:ht-ovw}

Simpson first noted a relation between Hindman's theorem and the Carlson-Simpson lemma~\cite{Carlson1984dual}. In this section, we give a formal counterpart to his observation by giving a simple proof of $\ovw{2}{\ell}$ using the Finite Union Theorem, a statement known to be equivalent to Hindman's theorem. A variation of the proof below was used by Dzhafarov et al.~\cite{Dzhafarov2017Effectiveness} to give an upper bound to the Open Dual Ramsey's theorem. A direct combinatorial proof of $\ovw{2}{\ell}$ in $\rca + \aca$ will be given in the next section. 

\begin{definition}
An \emph{IP collection} is an infinite collection of finite sets $\Ical \subseteq \Pcal_{fin}(\Nb)$
which is closed under \emph{non-empty} finite unions and contains an infinite subcollection
of pairwise disjoint sets.
\end{definition}

Note that any IP collection $\Ical$ necessarily contains an infinite $\Ical$-computable sequence $S_0 < S_1 < \dots$.

\begin{statement}[Finite union theorem]
For every $\ell \in \Nb$, $\fut_\ell$ is the statement ``For every coloring $c : \Pcal_{fin}(\Nb) \to \ell$,
there is a monochromatic IP collection''. $\wfut^2_\ell$ is the statement ``For every coloring $c : \Pcal_{fin}(\Nb) \times \Nb \to \ell$,
there is an IP collection $\Ical$ and a color $i < \ell$ such that $c(S, \min T) = i$ for every $S < T \in \Ical$.''
\end{statement}

\begin{theorem}
$\rca \vdash \forall \ell(\fut_\ell \to \wfut^2_\ell)$.
\end{theorem}
\begin{proof}
Assume $\ell \geq 2$, the other cases being trivial.
Let $f : \Pcal_{fin}(\Nb) \times \Nb \to \ell$ be an instance of $\wfut^2_\ell$.
Note that over $\rca$, $\fut_\ell \imp \aca$ and $\aca \imp \coh$.
Let $\vec{R}$ be a sequence of set defined for every $S \in \Pcal_{fin}(\Nb)$ and $i < \ell$
by $R_{S,i} = \{ n \in \Nb : f(S, n) = i \}$.
Apply $\coh$ to $\vec{R}$ to obtain an infinite $\vec{R}$-cohesive set $C$.
In particular, for every $S \in \Pcal_{fin}(\Nb)$, $\lim_{n \in C} f(S, n)$ exists.

Let $h : \omega \to C$ be a computable bijection.
Let $\tilde{f} : \Pcal_{fin}(\Nb) \to \ell$ be defined by $\tilde{f}(S) = \lim_{n \in C} f(h[S], n)$.
$\tilde{f}$ is a $\Delta^{0,f \oplus C}_2$ instance of $\fut_\ell$, so by the finite union theorem, there is an IP collection $\Ical \subseteq \Pcal_{fin}(\Nb)$.
and a color $i < \ell$ such that 
for every $S \in \Ical$, $\tilde{f}(S) = \lim_{n \in C} f(h[S], n) = i$. Note that for every $S \in \Ical$, $\min h[S] \in C$.
Therefore, by $f$-computably 
thinning-out the set $\Ical$, we obtain an IP collection $\Jcal \subseteq \Ical$
such that for every $S < T \in \Jcal$, $f(h[S], \min h[T]) = i$.
The set $\{h[S] : S \in \Jcal\}$ is a solution to $f$.
%
\end{proof}

\begin{theorem}
$\rca \vdash \forall \ell(\wfut^2_\ell \to \ovw{2}{\ell})$.
\end{theorem}
\begin{proof}
Let $f : 2^{<\omega} \to \ell$ be an instance of $\ovw{2}{\ell}$.
Define an instance $g : \Pcal_{fin}(\Nb) \times \Nb \to \ell$ of $\wfut^2_\ell$ as follows:
Given some $S \in \Pcal_{fin}(\Nb)$ and $n \in \Nb$, if $\max S < n$, then set $g(S, n) = f(\sigma)$,
where $\sigma$ is the binary string of length $n$ defined by $\sigma(i) = 1$ iff $i \in S$.
If $n \leq \max S$, set $g(S, n) = 0$. By $\wfut^2_\ell$, there is an IP collection $\Ical$
and a color $i < \ell$ such that $g(S, \min T) = i$ for every $S < T \in \Ical$.
Compute from $\Ical$ an infinite increasing sequence of pairwise disjoint finite sets
$F_0 < F_1 < \dots$ Let $W$ be the infinite variable word defined by 
$$
W(n) = \left\{\begin{array}{ll}
	1 & \mbox{ if } n \in F_0\\
	x_i & \mbox{ if } n \in F_i \mbox{ for some } i \geq 1\\
	0 & \mbox{ otherwise}	
\end{array}\right.
$$
The variable word $W$ and the sequence of the $F$'s is a solution to the instance $f$ of $\ovw{2}{\ell}$.
\end{proof}


\begin{corollary}
$\rca \vdash \aca^{+} \to \forall \ell \ovw{2}{\ell}$.
\end{corollary}
\begin{proof}
Immediate since $\aca^{+} \imp \forall \ell \fut_\ell \imp \forall \ell \wfut^2_\ell \imp \forall \ell \ovw{2}{\ell}$ over~$\rca$.
\end{proof}

\section{A proof of the Ordered Variable Word theorem in ACA}\label{sect:ovw-aca}

The proof of the previous section gave a very coarse computability-theoretic upper bound of the Ordered Variable Word theorem in terms of $\omega$-jumps. In this section, we give a direct combinatorial proof of $\ovw{2}{\ell}$ in $\rca + \aca$. Actually, every PA degree relative to $\emptyset'$ is sufficient to compute a solution of a computable instance of $\ovw{2}{\ell}$.
	We thereby answer a question of Miller and Solomon~\cite{Miller2004Effectiveness}.

\begin{theorem}\label{ovwth2}
For every $\ell \in\omega$,
every computable instance $c$ of $\ovw{2}{\ell}$,
every PA degree over $\emptyset'$ computes a solution to $c$.
\end{theorem}

A formalization of Theorem~\ref{ovwth2} yields
a proof of Theorem~\ref{thm:aca-ovw2}.

\begin{proof}[Proof of Theorem~\ref{thm:aca-ovw2}]
The proof of Theorem~\ref{ovwth2} can be
formalized within $\rca + \aca$. Indeed, the arguments require only arithmetical induction to be carried out, and
every model of $\rca + \aca$ is a model of
the statement ``For every set $X$, there is a set of PA degree over the jump of $X$.''
\end{proof}

Let us first introduce some notation.
For a finite set $F$ and a string $\sigma \in 2^{<\omega}$
let $\sigma_F$ be the binary string of length $|\sigma|$
defined by $\sigma_F(i) = \sigma(i)$ if $i \not \in F$, and $\sigma_F(i) = 1-\sigma(i)$ otherwise. Let $\leq_{lex}$ denote the shortlex order
on  $\omega^{<\omega}$, that is, the order with the shortest length first, and with the strings of same length sorted lexicographically.

In what follows, fix a coloring $\t{c}:2^{<\omega}\rightarrow \ell$,
and a string $\rho\in 2^{<\omega}$.

The main combinatorial lemma we use is
 Lemma~\ref{ovwprop1}.
As a warm up, we first prove the following
lemma \ref{ovwprop0}, which is a consequence of Lemma~\ref{ovwprop1}
and the proof is somehow similar but
much simpler. In the following lemma, one may think of $\rho_{P'}$ as a finite variable word, where the positions at $\t{P}$ are replaced by a same variable kind.

\begin{lemma}\label{ovwprop0}
For any
$P\subseteq \{0,\cdots,|\rho|-1\}$ with
$(\forall n\in P)[\rho(n) = 0]
\wedge |P|\geq \ell$,
there exist two
subsets $P'<\t{P}$ of $P$
with $\t{P}\ne\emptyset$
such that $\t{c}(\alpha) = \t{c}(\alpha_{\t{P}})$ where $\alpha = \rho_{P'}$.
\end{lemma}
\begin{proof}
Suppose $P = \{p_0 < \cdots < p_{m-1}\}$.
Let $\ell_0, \dots, \ell_m$ be defined by
$\ell_i = \t{c}(\rho_{\{p_0, \dots, p_{i-1}\}})$. In particular, $\ell_0 = \t{c}(\rho)$.
Since $|P| = m \geq \ell$,
so among $\ell_0,\cdots, \ell_m$,
there must exists $i<j$ such
that $\ell_i=\ell_j$.
Let $P' = \{p_0,\cdots,p_{i-1}\}$
(if $i=0$ then $P' = \emptyset$), and
$\t{P} = \{p_i,\cdots,p_{j-1}\}$,
let $\alpha = \rho_{P'}$.
Clearly $P'<\t{P}$ and $\t{P}\ne\emptyset$.
It is also easy to see that
$\t{c}(\alpha) = \ell_i = \ell_j = \t{c}(\alpha_{\t{P}})$.
\end{proof}

We now prove a technical lemma used in the proof of our main combinatorial
lemma (Lemma~\ref{ovwprop1}). The sequence in the following lemma is obtained by a simple greedy algorithm, with finitely many resets.

\begin{lemma}\label{ovwclaim1}
There exists a nonempty set
of colors $L\subseteq \{0,1,\cdots,\ell-1\}$,
$|L|+1$ many sets of binary strings
$\Gamma_0 = \{\tau^{\eta}\}_{\eta\in L},
\Gamma_1=\{\tau^\eta\}_{\eta\in L^2},
\cdots,\Gamma_{|L|}=\{\tau^\eta\}_{\eta\in L^{|L|+1}}$,
such that, letting
$$
\t{\eta} =
 \underbrace{\max L *\max L *\cdots*\max L}_{|L|+1\text{ many }}
$$
 and letting $\t{\rho} = \tau^{\t{\eta}}*0$, the following holds:
\begin{enumerate}
\item $\rho\prec\Gamma_0$ and $
\tau^{\eta}\prec\tau^\beta\Leftrightarrow
\eta<_{lex}\beta$;
\item $\t{\rho}(|\tau|) = 0$
for all $\tau\in \Gamma_i,i\leq |L|$;

\item for all $i\leq |L|$, $\eta\in L^{i+1}$,
let $\eta_0\prec\eta_1\prec\cdots\prec\eta_{i-1}$ denote
all nonempty predecessors of $\eta$, let
$Q = \big\{|\tau^{\eta_0}|,|\tau^{\eta_1}|,
\cdots,|\tau^{\eta_{i-1}}|\big\}$
(if $i=0$ then $Q=\emptyset$),
then $\t{c}(\tau^{\eta}_Q) = \eta(i)$;

\item let $P = \{|\tau^\eta|\}_{\eta\in L^{\leq |L|+1}}$,
for all subset $Q$ of $P$, all $\tau\succeq\t{\rho}$,
$\t{c}\big(\tau_Q\big)
\in L$.
\end{enumerate}
Moreover,
$\Gamma_i,i\leq |L|$ is  computable in the jump of $\t{c}$,
 uniformly in $\rho$.

\end{lemma}
\begin{proof}
We firstly show how to find $\Gamma_0$.
Start with $L = \{0,1,\cdots,\ell-1\}$.
At step 1, try to find a string $\tau\in 2^{<\omega}$ such that
$\t{c}(\rho\tau) =  0$ and
let $\tau^0 = \rho\tau$. Then
try to find a $\tau$ such that
$\t{c}(\tau^00\tau) = 1$ and
let $\tau^1 = \tau^00\tau$. Generally, after $\tau^j$ is found,
try to find $\tau$ such that
$\t{c}(\tau^j0\tau)=j+1$
and let $\tau^{j+1} = \tau^j0\tau$ if
$\tau$ is found.
If during the above process, after $\tau^{j-1}$ is
defined ( $\tau^{-1} = \rho$ ),
there is no $\tau$ such that
$\t{c}(\tau^{j-1}0\tau) = j$,
then we start all over again
with $\rho$ replaced by $\rho_1 =
\tau^j0$ and with $L$ replaced by $ L \setminus \{j\}$.

%
%

Generally, given a set of colors $L$ and after $\tau^{\beta}$ is found,
let $\eta$
be the immediate successor (with respect to $\leq_{lex}$ order restricted to $L$-strings)
 of $\beta$,
let $\eta_0\prec\eta_1\prec\cdots\prec\eta_{i-1}$ denote
all nonempty predecessors of $\eta$, let
$Q = \big\{|\tau^{\eta_0}|,|\tau^{\eta_1}|,
\cdots,|\tau^{\eta_{i-1}}|\big\}$
(if $i=0$ then $Q=\emptyset$),
we try to find $\tau$ such that
$\t{c}((\tau^{\beta}0\tau)_{Q}) = \eta(|\eta|-1)$.
If such a string $\tau$ does not exists then we
start  all over again with $\rho $
replaced by  $\tau^\beta0_{Q}$ and $L$ replaced by $L \setminus \{\eta(|\eta|-1)\}$.
If such $\tau$ exists then let $\tau^\eta =
\tau^\beta0\tau$.

Note that we have to
start over for at most $\ell-1$ times before we
ultimately succeed since
there are $\ell$ colors in total.
It is plain to check all the four items.
Also note that the sequence $\Gamma_0,\cdots,\Gamma_{|L|}$ is
$\t{c}'$-computable since we only need to
use the jump of $\t{c}$ to know whether the next $\tau^{\eta}$ can
 be found.

\end{proof}

\begin{lemma}\label{ovwprop1}

There exists a string $\t{\rho}\succ\rho$
 and a finite set $P\subseteq
 \big\{|\rho|,\cdots, |\t{\rho}|-1\big\}$ with
 $(\forall i\in P)[\t{\rho}(i) = 0]$
 such that
 for all $\sigma\succeq \t{\rho}$ there exists
 two subsets $P'<\t{P}$ of $P$ with $\t{P}\ne\emptyset$
 such that, letting $\alpha = \sigma_{P'}$,
$\t{c}(\alpha) = \t{c}(\alpha_{\t{P}}) = \t{c}(\alpha
 \uhr {\min \t{P} })$.
 Moreover, $|P|< \ell^{\ell+2}$, and $\t{\rho}, P$,
are computable in the jump of $\t{c}$, uniformly in $\rho$.
\end{lemma}
\begin{proof}
Let $L$ and $\t{\rho}$ satisfy Lemma~\ref{ovwclaim1}.
We claim that $\t{\rho}$ and $P = \{|\tau^\eta|\}_{\eta\in L^{\leq |L|+1}}$
satisfy the current lemma.
It is clear
by item 1 of Lemma~\ref{ovwclaim1} that $\t{\rho}\succ\rho$
and by item 2 of  Lemma~\ref{ovwclaim1} that
$(\forall i\in P)[\t{\rho}(i) = 0]$.

Fix an arbitrary $\sigma\succeq\t{\rho}$.
We now describe how to construct $P'$ and $\t{P}$.
Define $\ell_0, \dots, \ell_{|L|}$
and $p_0, \dots, p_{|L|}$ inductively by
$\ell_0 = \t{c}(\sigma)$,
$\ell_{i+1} = \t{c}(\sigma_{\{p_0,p_1,\cdots,p_i\}})$, and
$p_{i} =|\tau^{\ell_0\cdots \ell_{i}}|$ (where $\tau^{\ell_0\cdots \ell_{i}} \in \Gamma_i$).
%
%
%
%
%
%
Since $\ell_0,\cdots,\ell_{|L|}\in L$ (by item 4 of Lemma~\ref{ovwclaim1}), there is some $i<j\leq |L|$
such that $\ell_i=\ell_j $.
Let $P' = \{p_0,\cdots,p_{i-1}\}$
(if $i=0$ then $P' = \emptyset$),
$\t{P}= \{p_i,\cdots,p_{j-1}\}$,
and let $\alpha = \sigma_{P'}$. We claim that
$\t{c}(\alpha) = \t{c}(\alpha_{\t{P}}) = \t{c}(\alpha\uhr \min \t{P})$.
Note that $\min \t{P}  =p_i= |\tau^{\ell_0\cdots \ell_i}|$.
Therefore
$\alpha\uhr \min \t{P} = \tau^{\ell_0\cdots \ell_i}_{P'}$.
By item 3 of Lemma~\ref{ovwclaim1},
we have $\t{c}(\tau^{\ell_0\cdots \ell_i}_{P'}) = \ell_i $.
Meanwhile, by definition of $\ell_i$,
$\t{c}(\sigma_{P'}) =
\t{c}(\alpha) = \ell_i$.
By definition of $\ell_j$,
$\t{c}(\sigma_{P'\cup \t{P}}) =
\t{c}(\alpha_{\t{P}}) = \ell_j$.
Thus, $\t{c}(\alpha) = \t{c}(\alpha_{\t{P}})
 = \t{c}(\alpha\uhr \min \t{P} )$.

\end{proof}

We say that $(\t{\rho}, P)$ is \emph{$\t{c}$-valid}
if $P$ and $\t{\rho}$ satisfy Lemma~\ref{ovwprop1}.
We say that $(P',\t{P})$ \emph{witnesses $\t{c}$-validity of $(\t{\rho}, P)$ for $\sigma \succeq \t{\rho}$} if $P' < \t{P}  \subseteq P$, and letting $\alpha = \sigma_{P'}$,
$\t{c}(\alpha) = \t{c}(\alpha_{\t{P}}) = \t{c}(\alpha
 \uhr {\min \t{P} })$.
Before proving Theorem~\ref{ovwth2}, we start with
the following simpler version.

\begin{theorem}\label{ovwth1}
For every $\ell\in\omega$, every computable instance  $c:2^{<\omega}\rightarrow
\ell$ of $\ovw{2}{\ell}$,
every $PA$ degree over $\emptyset''$ computes
a solution to $c$.
\end{theorem}
\begin{proof}
It suffices to compute, given a PA degree relative to $\emptyset''$,
an infinite binary sequence $Y\in 2^\omega$
together with a sequence of finite
sets $\t{P}_0<\t{P}_{1}<\cdots$ with
$(\forall i\in\omega)(\forall n\in \t{P}_i)
[Y(n) = 0]$ such that the following holds:
\begin{quote}
Let $Position = \big\{\min \t{P}_i : i \geq 1 \big\}$.
There is some $\t{\ell}<\ell$ such that for all subset $J$ of $\omega$, letting
 $\t{P}_J = \bigcup\limits_{i\in J}\t{P}_i$, then we have,
 $
(\forall p\in Position)\big[
c(Y_{\t{P}_J}\uhr p) = \t{\ell}\ ].
$ 	
\end{quote}

Using Lemma~\ref{ovwprop1}, we first construct
a $\emptyset'$-computable sequence of strings
 $\t{\rho}_0 \prec \t{\rho}_1 \prec\cdots$,
a sequence of finite sets $P_i\subseteq \big\{
|\t{\rho}_{i-1}|,\cdots, |\t{\rho}_i|-1
\big\}$ and a sequence of colorings $c_i:[\t{\rho}_i]^\preceq\rightarrow
L_i$ inductively as follows.
$\t{\rho}_0 = \varepsilon$ and $c_0 = c$.
Given $\t{\rho}_i$ and $c_i:[\t{\rho}_i]^\preceq\rightarrow
L_i$, let $\t{\rho}_{i+1} \succeq \t{\rho}_i$
and $P_i\subseteq
\big\{ |\t{\rho}_i|,\cdots,|\t{\rho}_{i+1}|-1\big\}$ be such that
$(\t{\rho}_{i+1}, P_i)$- is $c_i$-valid,
and let $c_{i+1}$ be the coloring of $[\t{\rho}_{i+1}]^{\preceq}$ which on $\sigma \succeq \t{\rho}_{i+1}$
associates $\langle P', \t{P}, j \rangle$ such that
$(P',\t{P})$ witnesses $c_i$-validity of $(\t{\rho}_{i+1}, P_i)$
for $\sigma$, and $c_i(\sigma_{P'}) = j$. If there are multiple such tuples, take the least one, in some arbitrary order. Note that the range of $c_i$ is some finite set $L_i$.

We now analyze for $\sigma\succeq \t{\rho}_i$
 what $c_i(\sigma)= \langle P',\t{P},j \rangle$ means.
Note that elements of $L_i,i\in\omega$
admit a natural partial order $\lhd$ as follows:
for
$\langle P'_{0},\t{P}_0,j_0 \rangle\in L_{i},
\langle P'_{1},\t{P}_1,j_1\rangle \in L_{i+1}$,
$\langle P'_{1},\t{P}_1,j_1 \rangle $ is an immediate
successor of $\langle P'_{0},\t{P}_0,j_0 \rangle$
if and only if $j_1 = \langle P'_{0},\t{P}_0,j_0\rangle$.
Clearly every $j\in L_i$ admit a unique
immediate predecessor.

\begin{claim}\label{ovwclaim2}
Fix some $n \geq 1$, and let
$
\t{\ell} \lhd \langle P'_0,\t{P}_0,j_0\rangle \lhd \dots \lhd \langle P'_{n-1},\t{P}_{n-1},j_{n-1}\rangle = c_n(\sigma)
$,
Let $P' = \bigcup_{i\leq n-1}P'_i$ and
$\alpha = \sigma_{P'}$.
Then for any subset $J$ of $\{0,\cdots, n-1\}$,
$$(\forall p\in \big\{ \min \t{P}_j : 1 \leq j \leq n-1\big\}\cup\{|\alpha|\})
\big[\ c(\alpha_{\t{P}_J}\uhr p)
 = \t{\ell}\ \big].$$
\end{claim}
\begin{proof}
First we prove the claim for $n=1$.
By definition of $c_1(\sigma) = \langle P'_0,\t{P}_0,j_0 \rangle$,
letting $\beta = \sigma_{P'_0}$,
$c_0(\beta) = c_0(\beta_{\t{P}_0})= j_0=\t{\ell}$.
In other words,
 for any subset $J\subseteq \{0\}$,
$$(\forall p\in \big\{ \min\{\t{P}_j : 1 \leq j \leq 0\}\big\} \cup\{|\beta|\})
\big[\ c(\beta_{\t{P}_J}\uhr p)
 = \t{\ell}\ \big].$$
So the claim holds for $n=1$.
Suppose now the claim
 holds for $n-1$.

Suppose $c_n(\sigma) = \langle P'_{n-1},\t{P}_{n-1},j_{n-1} \rangle$. Let $\beta = \sigma_{P'_{n-1}}$. We have $c_{n-1}(\beta) =c_{n-1}(\beta_{\t{P}_{n-1}}) = c_{n-1}(\beta\uhr \min \t{P}_{n-1} ) = j_{n-1} = \langle P'_{n-2},\t{P}_{n-2},j_{n-2}\rangle$. As $c_{n-1}(\beta) = \langle P'_{n-2},\t{P}_{n-2},j_{n-2}\rangle$ and as $\t{\ell} \lhd \langle P'_{n-2},\t{P}_{n-2},j_{n-2}\rangle$, by induction hypothesis, for any subset $J$ of $\{0,\cdots,n-2\}$ we have:
\begin{align}\label{ovweq1}
&c(\beta_{(\cup_{i\leq n-2} P_i')\cup \t{P}_J}) = \t{\ell}.
\end{align}

Let $\beta' = \beta_{\t{P}_{n-1}}$. As $c_{n-1}(\beta') = \langle P'_{n-2},\t{P}_{n-2},j_{n-2}\rangle$ and as $\t{\ell} \lhd \langle P'_{n-2},\t{P}_{n-2},j_{n-2}\rangle$, by induction hypothesis, for any subset $J$ of $\{0,\cdots,n-2\}$ we have:
\begin{align}\label{ovweq2}
&c(\beta'_{(\cup_{i\leq n-2} P_i')\cup \t{P}_J}) = \t{\ell}.
\end{align}

As $c_{n-1}(\beta\uhr \min \t{P}_{n-1} ) = \langle P'_{n-2},\t{P}_{n-2},j_{n-2}\rangle$ and as $\t{\ell} \lhd \langle P'_{n-2},\t{P}_{n-2},j_{n-2}\rangle$, by induction hypothesis, for any subset $J$ of $\{0,\cdots,n-2\}$ we have:
\begin{align}\label{ovweq4}
&(\forall p \in\big\{ \min \t{P}_j : 1\leq j\leq n-2\big\}\cup \big\{|\beta\uhr \min\t{P}_{n-1}|\big\})
\big[\ c(\beta_{(\cup_{i\leq n-2} P_i')\cup\t{P}_J}\uhr p) = \t{\ell}\ \big].
\end{align}
But $|\beta\uhr \min\t{P}_{n-1}| = \min \t{P}_{n-1}$. So (\ref{ovweq4}) means for any subset $J$ of $\{0,\cdots,n-2\}$ we have:
\begin{align}\nonumber
&(\forall p \in\big\{ \min \t{P}_j : 1\leq j\leq n-1\big\})
\big[\ c(\beta_{(\cup_{i\leq n-2} P_i')\cup\t{P}_J}\uhr p) = \t{\ell}\ \big].
\end{align}
Or equivalently, for any subset $J$ of $\{0,\cdots,n-1\}$ we have:
\begin{align}\label{ovweq3}
&(\forall p \in\big\{ \min \t{P}_j : 1\leq j\leq n-1\big\})\big[\ c(\beta_{(\cup_{i\leq n-2} P_i')\cup\t{P}_J}\uhr p) = \t{\ell}\ \big].
\end{align}

Now from \ref{ovweq1}, \ref{ovweq2} and \ref{ovweq3} we deduce that for any subset $J$ of $\{0,\cdots,n-1\}$ we have:
$$(\forall p\in \big\{ \min \t{P}_j : 1 \leq j \leq n-1\big\}\cup\{|\beta|\})\big[\ c(\beta_{(\cup_{i\leq n-2} P_i')\cup\t{P}_J}\uhr p) = \t{\ell}\ \big]$$
which completes the proof of the claim since
$\beta_{\cup_{i\leq n-2} P_i'} = \alpha$.

\end{proof}

Let $\mcal{T}_0$ be the $\emptyset'$-computable set of all $\gamma$ such that
$(\forall i\leq |\gamma|)[\gamma(i)\in L_i]$,
$\gamma(i) \lhd \gamma(i+1)$ and $\gamma(|\gamma|-1) = c_{|\gamma|-1}(\t{\rho}_{|\gamma|})$.
Then, let $\mcal{T}$ be the downward closure of the set $\mcal{T}_0$
by the prefix relation. The tree $\mcal{T}$ is infinite by construction of the strings $\t{\rho_i}$, the colors $c_i$ and the sets $P_i$ : a witness for the $c_i$-validity of $(\t{\rho}_{i+1}, P_{i+1})$ for $\rho_{i+1}$ yields a node of $\mcal{T}_0$ of length $i+2$. The tree $\mcal{T}$ is also $\emptyset'$-computably bounded, and $\emptyset''$-computable.
Let $j_0*\langle P'_0,\t{P}_0,j_0\rangle *\langle P'_1,\t{P}_1,j_1\rangle *\cdots$
be an infinite path through $\mcal{T}$ computed by any PA degree over $\emptyset''$.
By construction, $\langle P'_i,\t{P}_i,j_i\rangle \lhd \langle P'_{i+1},\t{P}_{i+1},j_{i+1} \rangle$.
Let $X = \bigcup_{i\in\omega}\t{\rho}_i$,
$P' = \bigcup_{i\in\omega} P'_i$
and let $Y = X_{P'}$.
Clearly  $(\forall i \forall n\in \t{P}_i)[Y(n) = 0]$
and $Y$ is computable in the given PA degree relative to $\emptyset''$.
Therefore,
letting $Position = \big\{\min \t{P}_i : i \geq 1\big\}$, it suffices
to show that
for all subsets $J$ of $\omega$,
$$
(\forall p\in Position)\big[
c(Y_{\t{P}_J}\uhr p) = j_0\ ].
$$

Without loss of generality,
suppose $p = \min \t{P}_n$
and $J\subseteq \{0,\cdots,n-1\}$.
Since $j_0*\langle P'_0,\t{P}_0,j_0\rangle*\langle P'_1,\t{P}_1,j_1\rangle*\cdots
 \langle P'_n,\t{P}_n,j_n\rangle$ is an initial segment of some element in
 $\mathcal{T}_0$,
there must exist some
$N>n$
 such that
 $c_N(\t{\rho}_{N+1}) = \langle P'_{N-1},\t{P}_{N-1},j_{N-1}\rangle$.
 Let $\sigma = \t{\rho}_{N+1}, \alpha =\sigma_{P'}$. Clearly
 $\alpha\prec Y\wedge |\alpha|>p$. Moreover,
 by Claim~\ref{ovwclaim2},
 $c(\alpha_{\t{P}_J}\uhr p) = j_0$.
 Thus $c(Y_{\t{P}_J}\uhr p) = j_0$.

\end{proof}

Finally, we slightly modify the proof of Theorem~\ref{ovwth1}
to derive Theorem~\ref{ovwth2}.

\begin{proof}[Proof of Theorem~\ref{ovwth2}]
The main point is to make the tree
$\mcal{T}$ $\emptyset'$-computable.
To ensure this, after we obtain
$\t{\rho}_i,c_i$, we do not
directly go to $\t{\rho}_{i+1}$.
Instead, we $\emptyset'$-compute
$\t{\rho}_i^{0}\prec\t{\rho}_i^1\prec\cdots
\prec\t{\rho}_i^{r_i}$
such that $\t{\rho}_i^0\succ\t{\rho}_i$ and
$c_i\big(
\{\tau:\tau\succeq \t{\rho}_i^{r_i} \}\big)
 \subseteq c_i\big(\big\{\t{\rho}_i^0,\cdots,\t{\rho}^{r_i}_i \big\}
 \big)$. Then we $\emptyset'$-compute $\t{\rho}_{i+1}\succ
 \t{\rho}_i^{r_i}$ as in the proof of Theorem~\ref{ovwth1}.
 Note that this indeed can be
 achieved using $\emptyset'$ since
 $c_i$ is computable.
 Define $\mcal{T}$ to be
 the set of all $\gamma$ such that
 $(\forall i\leq |\gamma|)[\gamma(i)\in L_i]$,
$\gamma(i) \lhd \gamma(i+1)$,
 and either $|\gamma| = 1\wedge \gamma\in L_0$ or
 there exists $\t{\rho}_{|\gamma|-1}^u$
 with $c_{|\gamma|-1}(\t{\rho}_{|\gamma|-1}^u)
 =\gamma(|\gamma|-1)$.
 It is easy to see
that $\mcal{T}$ is $\emptyset'$-computable
since $c_i$ is computable for all $i$ and
the sequences $\langle c_i : i \in \omega \rangle$ and
$\langle \t{\rho}_i^v : i\in\omega, v\leq r_i \rangle$
are $\emptyset'$-computable.

Now we show that $\mcal{T}$ is a tree.
Suppose $\gamma\in \mcal{T}$,
$|\gamma| = n+1$ with $n \geq 1$, and  $c_n(\t{\rho}_n^u)
 =\gamma(n) = \langle P',\t{P},j \rangle \in L_n$.
 We claim that $\gamma \uhr n \in \mcal{T}$.
 If $n = 1$, then $\gamma \uhr 1 \in L_0 \subseteq \mcal{T}$.
 Otherwise, let
$\langle Q',\t{Q},k \rangle \in L_{n-1}$ be the predecessor
  of $\langle P',\t{P},j \rangle$.
We need to show that
there exists $\t{\rho}_{n-1}^v$ such that
$c_{n-1}(\t{\rho}_{n-1}^v) = \langle Q',\t{Q},k \rangle$.
$c_n(\sigma) = \langle P',\t{P},j \rangle$ implies
that, letting $\alpha = \sigma_{P'}$,
$c_{n-1}(\alpha) = c_{n-1}(\alpha\uhr \min \t{P})
 = j = \langle Q',\t{Q},k \rangle$.
Note that $\alpha\succeq \t{\rho}_{n-1}^{r_{n-1}}$
since $P'> |\t{\rho}_{n-1}^{r_{n-1}}|$.
But $c_{n-1}\big( \{\tau:\tau\succeq \t{\rho}_{n-1}^{r_{n-1}}\}\big)
\subseteq c_{n-1}\big(\{\t{\rho}_{n-1}^0,\cdots,\t{\rho}_{n-1}^{r_{n-1}}\} \big)$.
Therefore there exists $\t{\rho}_{n-1}^u$ such that
$c_{n-1}(\t{\rho}_{n-1}^u) = \langle Q',\t{Q},k \rangle$.
It follows that $\gamma \uhr n \in \mcal{T}$ and that $\mcal{T}$ is a tree.
Any PA degree relative to $\emptyset'$ computes an infinite path through $\mcal{T}$. The rest of the proof goes exactly the same
 as Theorem~\ref{ovwth1}.
\end{proof}

We now give an alternative proof of Theorem~\ref{ovwth2} based on the definitional complexity of the solutions of~$c$.

\begin{proof}[Second proof of Theorem~\ref{ovwth2}]
Let $P_0, P_1, \dots$ be the $\emptyset'$-computable sequence defined in the proof of Theorem~\ref{ovwth1}. We have seen that there exists an infinite ordered variable word such that the $n$th variable kind appears before the position $\max P_n$. Let $\mcal{T}$ be the tree of all finite ordered variable words which are finite solutions to $c$ and such that the $n$th variable appears before the position $\max P_n$. By the previous observation, the tree is infinite, $\emptyset'$-computable, and $\emptyset'$-computably bounded. Any PA degree relative to $\emptyset'$ computes an infinite variable word which, by construction of $\mcal{T}$, is a solution to $c$. This completes the proof of Theorem~\ref{ovwth2}.
\end{proof}

Note that the above proof can be slightly modified to obtain a proof of a sequential version of the ordered variable word.

\begin{statement}
$\mathsf{Seq}\ovw{n}{\ell}$ is the statement ``If $c_0, c_1, \dots$ is a sequence of $\ell$-colorings of a fixed alphabet $A$ of cardinality $n$, there exists a variable word $W$ such that for every $i \in \omega$ and every $\bar b \in A^i$, $\{W(\bar b\bar a) : \bar a \in A^{<\infty}\}$ is monochromatic for $c_i$.''
\end{statement}

\begin{theorem}
For every computable instance $c_0, c_1, \dots$ of $\mathsf{Seq}\ovw{2}{\ell}$,
every PA degree relative to $\emptyset'$ computes a solution to $\bar c$.
\end{theorem}
\begin{proof}
	The proof is similar to Theorem~\ref{ovwth2}. Using Lemma~\ref{ovwprop1}, we first construct a $\emptyset'$-computable sequence of strings
 $\t{\rho}_0 \prec \t{\rho}_1 \prec\cdots$,
a sequence of finite sets $P_i\subseteq \big\{
|\t{\rho}_{i-1}|,\cdots, |\t{\rho}_i|-1
\big\}$ and a sequence of colorings $d_i:[\t{\rho}_i]^\preceq\rightarrow
L_i$ inductively as follows.
$\t{\rho}_0 = \varepsilon$ and $d_0 = c_0$.
Given $\t{\rho}_i$ and $d_i:[\t{\rho}_i]^\preceq\rightarrow
L_i$, let $\t{\rho}_{i+1} \succeq \t{\rho}_i$
and $P_i\subseteq
\big\{ |\t{\rho}_i|,\cdots,|\t{\rho}_{i+1}|-1\big\}$ be such that
$(\t{\rho}_{i+1}, P_i)$- is $d_i$-valid,
and let $d_{i+1}$ be the coloring of $[\t{\rho}_{i+1}]^{\preceq}$ which on $\sigma \succeq \t{\rho}_{i+1}$
associates $\langle P', \t{P}, j, k \rangle$ such that
$(P',\t{P})$ witnesses $d_i$-validity of $(\t{\rho}_{i+1}, P_i)$
for $\sigma$, $d_i(\sigma_{P'}) = j$ and $c_{i+1}(\sigma_{P'}) = k$. Note that the main difference with the previous construction is that we handle more and more colorings among $c_0, c_1, \dots$ at each level. The remainder of the proof is the same as in Theorem~\ref{ovwth2}.
\end{proof}

The theorem above is optimal, in that we can obtain the following reversal.

\begin{theorem}
There is a computable instance $c_0, c_1, \dots$ of $\mathsf{Seq}\ovw{2}{2}$,
such that every solution is of PA degree relative to $\emptyset'$.
\end{theorem}
\begin{proof}
Let $R_0, R_1, \dots$ be a uniformly computable sequence of sets
such that for every $e$,
if $\Phi^{\emptyset'}_e(e) \downarrow = 0$ then $R_e$ is finite,
and if $\Phi^{\emptyset'}_e(e) \downarrow = 1$ then $R_e$ is cofinite.
In particular, any function $f : \omega \to 2$ such that $f(e)$ gives
a side of $R_e$ which is infinite, is DNC$_2$ relative to $\emptyset'$,
hence of PA degree relative to $\emptyset'$.
Let $c_i : 2^{<\infty} \to 2$ be defined by $c_i(\sigma) = 1$ iff $|\sigma| \in R_i$, and let $W$ be a solution to $\bar c$, that is, a variable word $W$ such that for every $i \in \omega$ and every $\bar b \in A^i$, $\{W(\bar b\bar a) : \bar a \in A^{<\infty}\}$ is monochromatic for $c_i$.
We claim that $W$ computes such a function $f$.
Given $e$, let $f(e) = c_e(W(\bar b))$, where $\bar b  \in 2^e$ is arbitrary (this is well-defined, since $c_e(\bar b)$ depends only on the length of $\bar b$). By definition of $W$, $\{W(\bar b\bar a) : \bar a \in A^{<\infty}\}$ is monochromatic for $c_e$, the color of $c_e(W(\bar b))$ appears infinitely often in $R_e$. Therefore, $W$ is of PA degree relative to $\emptyset'$. This completes the proof.
\end{proof}

\section{A difficult instance of the Ordered Variable Word theorem}\label{sect:ovw-lower-bounds}

Miller and Solomon~\cite{Miller2004Effectiveness} constructed a computable instance of $\ovw{2}{2}$ with no $\Delta^0_2$ solution. In this section, we strengthen their proof by constructing a computable instance of $\ovw{2}{2}$ such that every solution is of DNC degree relative to $\emptyset'$, using a significantly simpler argument.

The proof makes an essential use of a computable version of Lovasz Local Lemma proven by Rumyantsev and Shen~\cite{Rumyantsev2014Probabilistic}. The idea of using Lovasz Local Lemma to analyse the computability-theoretic strength of problems in reverse mathematics comes from Csima and Dzhafarov, Hirschfeldt, Jockusch, Solomon and Westrick~\cite{Csima2018reverse}, who proved that a version of Hindman's theorem for subtractions is not computably true.

\begin{definition}
	Fix a countable set of variables $x_0, x_1, \dots$
	A \emph{(disjunctive) clause} $C$ is a tuple of the form $(x_{n_1} = i_1 \vee \dots \vee x_{n_k} = i_k)$, with $i_1, \dots, i_k < 2$. The \emph{length} of $C$ is the integer $k$. An \emph{infinite CNF formula} is an infinite conjunction of disjunctive clauses. An infinite CNF formula $\bigwedge_n C_n$ is \emph{computable} if the function which given $n$ outputs a code for $C_n$ is computable, and the set of $n$ such that $C_n$ contains the variable $x_j$ is uniformly computable in $j$.
\end{definition}

\begin{theorem}[{Rumyantsev and Shen~\cite{Rumyantsev2014Probabilistic}}]\label{thm:lll-computable}
For every $\alpha \in (0,1)$, there exists some $N \in \omega$ such that every computable infinite CNF where each variable appears in at most $2^{\alpha n}$  clauses of size $n$ (for every n) and all clauses have size at least $N$, has a computable satisfying assignment.
\end{theorem}

\begin{theorem}\label{thm:ovw-delta2}
There is a computable instance $c$ of $\ovw{2}{2}$ and a computable function $h : \omega \to \omega$ such that if $\Phi_e^{\emptyset'}$ outputs a finite variable word in which the first $h(e)$ variable kinds occur, then $\Phi_e^{\emptyset'}$ is not extendible into an infinite solution to $c$.
\end{theorem}
\begin{proof}
Fix $\alpha = 0.5$, and let $N$ be the threshold of Theorem~\ref{thm:lll-computable}. For every index $e$ and stage $s$, we interpret $\Phi^{\emptyset'}_e[s]$ as a finite variable word $W_{e,s}$ with exactly $N+e$ variable kinds, and where a new variable occurs right after $W_{e,s}$. Such a variable word induces a binary tree $T_{e,s}$ with $2^{N+e}$ leaves. Let $L_{e,s}$ be the set of leaves of $T_{e,s}$, that is, the set of all instantiations of the variable word $W_{e,s}$. Moreover, all the leaves of $T_{e,s}$ have the same length $n_{e,s}$.

The idea is the following: since the variable word is ordered and a new variable kind occurs right after $W_{e,s}$, no variable among the first $N+e$ variables can occur after $W_{e,s}$. If $W$ is a solution to $c$ with initial segment $W_e = \lim_s W_{e,s}$ for some color $i$, then $W$ must be homogeneous for $c$ for every instance of the variables, so in particular when setting all the variables after the $N+e$ first ones to 0. Hence, there must be infinitely many strings $\tau$ such that for every $\sigma \in \lim_s L_{e,s}$, $c(\sigma\tau) = i$. By ensuring that for cofinitely many $\tau$, there is some  $\sigma \in L_{e,|\tau|}$ such that $c(\sigma\tau) \neq i$, we force $W_e$ not to be a solution to $c$ for color $i$.

Fix a countable collection of variables $(x_\rho : \rho \in 2^{<\omega})$. Each variable $x_\rho$ corresponds to the color of the string $\rho$.
Given some $s\in\omega,\tau \in 2^{<\omega}$ and some $i < 2$,
if $n_{e,s}+|\tau| = s $, then
let $C_{e,s,\tau,i}$ be the disjunctive $2^{N+e}$-clause
$$
\bigvee \{x_{\sigma\tau} = i : \sigma \in L_{e,s} \}.
$$
And let $C$ be the conjunction
$$
\bigwedge\limits_{n_{e,s}+|\tau| = s }
 \{ C_{e,s,\tau, i} : e \in \omega, \tau \in 2^{<\omega}, i < 2\}.
$$
This infinite CNF formula is clearly computable. Clearly $C_{e,s,\tau,i}$ has length $2^{N+e}$. Note that for every $\rho,e$, there exists at most one $\tau$ such that $(\exists \sigma\in L_{e,|\rho|})[\sigma\tau = \rho]$. Therefore, each variable $x_\rho$ appears in at most $2$ clauses of length $2^{N+e}$, namely, $C_{e, |\rho|,\tau, 0}$ and $C_{e, |\rho|, \tau, 1}$, where $\tau$ is such that $(\exists \sigma\in L_{e,|\rho|})[\sigma\tau = \rho]$. Therefore, this formula satisfies the conditions of Theorem~\ref{thm:lll-computable}, and has a computable assignment $c : 2^{<\omega} \to 2$. By construction, letting $h(e) = N+e+1$, the formula ensures that if $\Phi_e^{\emptyset'}$ outputs a finite variable word in which the first $h(e)$ variables kinds occur, then $\Phi_e^{\emptyset'}$ is not extendible into an infinite solution to $c$.
\end{proof}

\begin{definition}
	A function $f : \omega \to \omega$ is \emph{diagonally non-computable relative to $X$} (or $X$-dnc) if for every $e$,
	$f(e) \neq \Phi_e^X(e)$.
\end{definition}

\begin{corollary}
	There is a computable instance $c$ of $\ovw{2}{2}$ such that every solution is of $\emptyset'$-dnc degree.
\end{corollary}
\begin{proof}
	Let $c$ and $h$ be as in Theorem~\ref{thm:ovw-delta2}. For every $e$, let $\alpha_e$ be a computable bijection from the finite variable words in which the first $h(e)$ variable kinds occur, to the set of the integers.
	By Kleene's fixpoint theorem, there is a computable function $g : \omega \to \omega$ such that for every $e$, $\Phi^{\emptyset'}_{g(e)} = \alpha^{-1}_{g(e)}(\Phi^{\emptyset'}_e(e))$.	
	
	Let $W$ be a solution to $c$, that is, an infinite variable word. Let $f$ be the $W$-computable function defined by $f(e) = \alpha_{g(e)}(w_e)$, where $w_e$ is the first initial segment of $W$ in which the first $h(g(e))$ variable kinds occur. We claim that $f$ is $\emptyset'$-dnc. Indeed, given $e \in \omega$, $w_e \neq \Phi^{\emptyset'}_{g(e)}$, so
	$$
	f(e) = \alpha_{g(e)}(w_e) \neq \alpha_{g(e)}(\Phi^{\emptyset'}_{g(e)})  = \Phi^{\emptyset'}_e(e)
	$$
	This completes our proof.
\end{proof}

We conclude this section with a small computational observation about $\vw{2}{2}$ based on the syntactical form of the statement.

\begin{definition}
	A function $g : \omega \to \omega$ \emph{dominates}
	$f : \omega \to \omega$ if $(\forall x)f(x) < g(x)$.
	A function $f : \omega \to \omega$ is \emph{hyperimmune}
	if it is not dominated by any computable function.
	A Turing degree is \emph{hyperimmune-free} if it does not contain any hyperimmune function.
\end{definition}

\begin{lemma}[Folklore]\label{lem:pi01-hi-pa}
Let $\Psf$ be a statement of the form $(\forall X)[\Phi(X) \rightarrow (\exists Y)\Psi(X, Y)]$ where $\Phi$ is an arbitrary predicate, and $\Psi$ is a $\Pi^0_2$ predicate. For every computable instance $I$ of $\Psf$, if $I$ has a solution of hyperimmune-free degree, then every PA degree computes a solution to $I$.
\end{lemma}
\begin{proof}
	Say $\Psi(X, Y) \equiv (\forall x)(\exists y)\Theta(X \uh y, Y \uh y, x, y)$, where $\Theta$ is a decidable predicate.
	Let $I$ be a computable $\Psf$-instance with a solution $S$ of hyperimmune-free degree. 
	Let $h : \omega \to \omega$ be the $S$-computable function such that for every $x$, $\Theta(I, S, x, h(x))$ holds. In particular, there is a computable function $g : \omega \to \omega$ such that $(\forall x)\max (h(x), S(x)) < g(x)$. Let $T \subseteq \omega^{<\omega}$ be the computably bounded tree defined by
	$$
	T = \left\{ \sigma \in \omega^{<\omega} : \begin{array}{l}
 		(\forall x < |\sigma|)\sigma(x) < g(x)) \wedge  \\
 		(\forall x < |\sigma|)[g(x) < |\sigma| \rightarrow (\exists y < |\sigma|)\Theta(I \uh y, \sigma \uh y, x, y)]
 \end{array} \right\}
	$$
	In particular, $S \in [T]$, so the tree is infinite. Moreover,
	any $R \in [T]$ is a solution to $I$, and any PA degree computes a member
	of $[T]$. This completes the proof.
\end{proof}

\begin{corollary}
	There is a computable instance of $\vw{2}{2}$ such that
	every solution is of hyperimmune degree.
\end{corollary}
\begin{proof}
	First, note that the statement $\vw{2}{2}$ is of the form of Lemma~\ref{lem:pi01-hi-pa}.
	Let $c : 2^{<\omega} \to 2$ be the computable
	instance of $\vw{2}{2}$ with no low solution constructed by Miller and Solomon~\cite{Miller2004Effectiveness} or by Theorem~\ref{thm:ovw-delta2}. Letting $\dbf$ be a low PA degree, $\dbf$ computes no solution to~$c$, hence by Lemma~\ref{lem:pi01-hi-pa}, every solution to~$c$ is of hyperimmune degree.
\end{proof}

It is still unknown whether there is a computable instance of $\ovw{2}{2}$ such that every solution is PA over $\emptyset'$, or even just computes $\emptyset'$. In particular the following questions remain open:

\begin{question}
	Does $\vw{2}{2}$ or $\ovw{2}{2}$ imply $\aca$ over $\rca$?
\end{question}

\begin{question}
	Is there a computable instance of $\vw{2}{2}$ or $\ovw{2}{2}$ such that the measure of oracles computing a solution to it is null?
\end{question} 


\bibliographystyle{plain}
\bibliography{bibliography}

\begin{thebibliography}{10}

\bibitem{Carlson1984dual}
Timothy~J. Carlson and Stephen~G. Simpson.
\newblock A dual form of {R}amsey's theorem.
\newblock {\em Adv. in Math.}, 53(3):265--290, 1984.

\bibitem{Csima2009strength}
Barbara~F. Csima and Joseph~R. Mileti.
\newblock {The strength of the rainbow {R}amsey theorem}.
\newblock {\em Journal of Symbolic Logic}, 74(04):1310--1324, 2009.

\bibitem{Dzhafarov2017Effectiveness}
Reed Solomon Linda Brown~Westrick Damir D.~Dzhafarov, Stephen~Flood.
\newblock Effectiveness for the dual ramsey theorem.
\newblock To appear. Available at
  \url{http://www.math.uconn.edu/~damir/papers/dualRT.pdf}, 2017.

\bibitem{Friedman2000Issues}
Harvey Friedman and Stephen~G. Simpson.
\newblock Issues and problems in reverse mathematics.
\newblock In {\em Computability theory and its applications ({B}oulder, {CO},
  1999)}, volume 257 of {\em Contemp. Math.}, pages 127--144. Amer. Math. Soc.,
  Providence, RI, 2000.

\bibitem{Jockusch197201}
Carl~G. Jockusch and Robert~I. Soare.
\newblock {$\Pi^0_1$} classes and degrees of theories.
\newblock {\em Transactions of the American Mathematical Society}, 173:33--56,
  1972.

\bibitem{Miller2004Effectiveness}
Joseph~S. Miller and Reed Solomon.
\newblock Effectiveness for infinite variable words and the dual {R}amsey
  theorem.
\newblock {\em Arch. Math. Logic}, 43(4):543--555, 2004.

\bibitem{Montalban2011Open}
Antonio Montalb{\'a}n.
\newblock Open questions in reverse mathematics.
\newblock {\em Bulletin of Symbolic Logic}, 17(03):431--454, 2011.

\bibitem{Rumyantsev2014Probabilistic}
Andrei Rumyantsev and Alexander Shen.
\newblock Probabilistic constructions of computable objects and a computable
  version of lov{\'a}sz local lemma.
\newblock {\em Fundamenta Informaticae}, 132(1):1--14, 2014.

\bibitem{Simpson2009Subsystems}
Stephen~G. Simpson.
\newblock {\em {Subsystems of Second Order Arithmetic}}.
\newblock Cambridge University Press, 2009.

\bibitem{Slaman1997note}
Theodore~A. Slaman.
\newblock A note on dual ramsey theorem.
\newblock Unpublished, January 1997.

\bibitem{Towsner2012simple}
Henry Towsner.
\newblock A simple proof and some difficult examples for {H}indman's theorem.
\newblock {\em Notre Dame J. Form. Log.}, 53(1):53--65, 2012.

\end{thebibliography}

\appendix


\end{document}